\documentclass[reqno]{amsart}
\usepackage{amsmath, amsthm, amssymb, amstext}

\usepackage[left=3cm,right=3cm,top=3cm,bottom=3cm]{geometry}
\usepackage{hyperref,xcolor}
\hypersetup{
 pdfborder={0 0 0},
 colorlinks,
}
\usepackage{enumitem}
\newtheorem{theorem}{Theorem}
\newtheorem{remark}[theorem]{Remark}

\newtheorem{proposition}[theorem]{Proposition}
\newtheorem{corollary}[theorem]{Corollary}

\newtheorem{example}[theorem]{Example}

\DeclareMathOperator*{\esssup}{ess ~sup}         %

\newcommand{\N}{\mathbb{N}}

\newcommand{\R}{\mathbb{R}}

\newcommand{\RN}{\mathbb{R}^N}

\newcommand{\eps}{\varepsilon}

\newcommand{\Om}{\Omega}
\newcommand{\rand}{\partial\Omega}

\newcommand{\into}{\int_{\Omega}}

\renewcommand{\l}{\left}
\renewcommand{\r}{\right}

\numberwithin{theorem}{section}
\numberwithin{equation}{section}

\title[Implicit equations involving the $p$-Laplace operator]{Implicit equations involving the $p$-Laplace operator}

\author[G.\,Marino]{Greta Marino}
\address[G.\,Marino]{Technische Universit\"{a}t Chemnitz, Fakult\"{a}t f\"{u}r Mathematik, Reichenhainer Stra\ss{} e 41, 09126 Chemnitz, Germany}
\email{greta.marino@mathematik.tu-chemnitz.de}

\author[A.\ Paratore]{Andrea Paratore}
\address[A.\,Paratore]{Dipartimento di Matematica e Informatica, Universit\`a degli Studi di Catania, Viale A. Doria 6, I-95125 Catania, Italy}
\email{aparatore@dmi.unict.it}



\begin{document}

\begin{abstract}

In this work we study the existence of solutions $u \in W^{1,p}_0(\Omega)$ to the implicit elliptic problem $ f(x, u, \nabla u, \Delta_p u)= 0$ in $ \Omega $, where $ \Omega $ is  a bounded domain in $ \R^N $, $ N \ge 2 $, with smooth boundary $ \partial \Omega $, $ 1< p< \infty $, and $ f\colon \Omega \times \R \times \R^N \times \R \to \R $. We choose the particular case when the function $ f $ can be expressed in the form $ f(x, z, w, y)= \varphi(x, z, w)- \psi(y) $, where the function $ \psi $ depends only on the $p$-Laplacian $ \Delta_p u $. We also present some applications of our results.

\end{abstract}

\maketitle

\section{Introduction and main results}

Let $ \Omega \subset \R^N $, $ N \ge 2 $, be a bounded domain with smooth boundary $ \partial \Omega $, let $ 1< p< \infty $, let $Y \subseteq \R$ be a nonempty interval possibly coinciding with $\R$, and let $ f\colon \Omega \times \R \times \R^N \times \R \to \R $. In this paper, we shall consider the following implicit elliptic problem
	\begin{equation} 
	 \label{equazione}
	u \in W^{1,p}_0(\Omega), \qquad f(x, u, \nabla u, \Delta_p u)= 0 \quad \text{in} \, \, \, \Omega,
	\end{equation}
where $ \Delta_p $ denotes the $p$-Laplace operator, namely
	\[
	\Delta_p u:= \text{div}(|\nabla u|^{p-2} \nabla u) \quad \forall \, u \in W^{1,p}(\Omega).
	\]
We consider the special case  $ f(x, z, w, y)= \varphi(x, z, w)- \psi(y) $, with $ \varphi\colon  \Omega \times \R \times \R^N \to \R$  and  $ \psi \colon Y \to \R $. We require that $ \psi $ depends only on $ \Delta_p u $.  We further distinguish among the case when $ \varphi $ is a Carath\'eodory function depending on $ x, u $, and $ \nabla u $, and the case when $ \varphi $ is allowed to be highly discontinuous in each variable. In this last case, the dependence on the gradient is no more allowed.

In both situations we first reduce problem  \eqref{equazione} to an elliptic differential inclusion, but  methods used are different and  depend on the regularity of the function $ \varphi $ and on the structure of the problem. 

More precisely, in the first case we make use of a result in \cite{ricceri1} to obtain the inclusion
\begin{equation}
\label{inclusione}
-\Delta_p u \in F(x, u, \nabla u),
\end{equation}
where $ F $ is a lower semicontinuous selection of the multifunction
\[
(x, z, w) \mapsto \{y \in Y: \varphi(x, z, w)- \psi(y)= 0 \}.
\]
A function $ u \in W^{1,p}_0(\Omega) $ is called a (weak) solution to \eqref{inclusione} if  there exists $ v \in L^{p'}(\Omega) $,  $ p' $ being the conjugate exponent of $ p $, such that $ v(x) \in F(x, u(x), \nabla u(x)) $ for almost every $ x \in \Omega $ and
\[
\int_{\Omega} |\nabla u|^{p-2} \nabla u \cdot \nabla w dx= \int_{\Omega} v w dx \quad \forall \, w \in W^{1,p}_0(\Omega).
\]
We start with the general case  $ Y= \R $ and then we deduce, as a byproduct, the existence result when $ Y $ is a closed interval of $ \R $.

Existence of solutions to  \eqref{inclusione} is obtained by means of the following result, which  is based on a selection theorem for decomposable-valued multifunctions, see \cite{bart} and \cite{marano1}.

\begin{theorem}[Theorem 3.1 of \cite{marmos}]
\label{mosconi}
Let $ F \colon \Omega \times \R \times \R^N \to 2^{\R} $ be a closed-valued multifunction. Suppose that
\begin{enumerate}
\item[(h1)] $ F $ is $ \mathcal L(\Omega) \otimes \mathcal B(\R \times \R^N)$-measurable;

\item[(h2)] for almost every $  x \in \Omega $, the multifunction $ (z, w) \mapsto F(x, z, w) $ turns out to be lower semicontinuous;

\item[(h3)] there exist $ a \in L^{p'}(\Omega, \R^+_0), b, c \ge 0 $, with $ \frac{b}{\lambda_{1, p}}+ \frac{c}{\lambda_{1, p}^{1/p}}< 1$,  such that
\[
\inf_{y \in F(x, z, w)} |y|< a(x)+ b|z|^{p-1}+ c|w|^{p-1} \quad \text{in }  \Omega \times \R \times \R^N.
\]

\end{enumerate}

Then, \eqref{inclusione} has a solution $ u \in W^{1,p}_0(\Omega) $.

\end{theorem}

Here, $ \lambda_{1, p} $ is the first eigenvalue of the $p$-Laplacian in the space $ W^{1,p}_0 (\Omega) $. 

The following is our main result, which extends \cite[Theorem 3.2]{marano1} to the case $ p \ne 2 $.

\begin{theorem}

Let  $ \varphi \colon \Omega \times \R \times \R^N \to \R $ be a Carath\'eodory function and let $ \psi \colon \R \to \R $ be continuous. Suppose that
\begin{enumerate}

\item[(i)] $\psi $ is non-constant on intervals; 

\item[(ii)] for all $ (x, z, w)  \in \Omega \times \R \times \R^N $, the function $ y \mapsto \varphi(x, z,w)- \psi(y) $ changes sign;

\item[(iii)] there exist $ a \in L^{p'}(\Omega, \R^+_0), b, c \ge 0 $, with $ \frac{b}{\lambda_{1,p}}+ \frac{c}{\lambda_{1,p}^{1/p}}< 1 $, such that
\[
\sup \{\vert y \vert: y \in \psi^{-1}(\varphi(x, z, w)) \} < a(x)+ b|z|^{p-1}+ c|w|^{p-1}, 
\]
for all $ (x, z, w) \in \Omega \times \R \times \R^N $.

\end{enumerate}

Then, there exists $ u \in W^{1,p}_0(\Omega) $ such that
\begin{equation}
\label{equa}
\psi(-\Delta_p u)= \varphi(x, u, \nabla u) \quad  \text{in } \Omega .
\end{equation}

\end{theorem}

When $ \varphi $ is discontinuous we essentially follow \cite[Theorem 3.1]{marano4} to construct an appropriate upper semicontinuous multifunction $ F $ related with $\psi^{-1}$ and $\varphi$, and then we solve the elliptic differential inclusion $-\Delta_p u \in F(x,u)$ using the following

\begin{theorem}[Theorem 2.2 of \cite{marano2}]
\label{plap}
Let $ U $ be a nonempty set, let $ \Phi\colon U \to W^{1,p}_0(\Omega)$ and $ \Psi\colon U \to L^{p'}(\Omega) $ be two operators, and let $ F\colon \Omega \times \R \to 2^{\R} $ be a convex closed-valued multifunction. Suppose that
\begin{enumerate}

\item[($i_1$)] $ \Psi $ is bijective and  $ v_h \rightharpoonup v$ in $ L^{p'}(\Omega) $ implies, up to subsequences,  $ \Phi(\Psi^{-1}(v_h)) \to \Phi(\Psi^{-1}(v)) $ a.e. in $ \Omega $. Furthermore, a non-decreasing function $ g\colon \R^+_0 \to \R^+_0 \cup \{+\infty\} $ can be defined in such a way that 
\[
\|\Phi(u)\|_{\infty} \le g(\|\Psi(u)\|_{p'}) \quad \forall \, u \in U;
\]

\item[($i_2$)] $ F(\cdot \,, z) $ is measurable for all $ z \in \R $;

\item[($i_3$)] $ F(x, \cdot) $ has a closed graph for almost every $ x \in \Omega $;

\item[($i_4$)] There exists $ r> 0 $ such that the function 
\[
\rho(x):= \sup_{|z| \le g(r)} d(0, F(x,z)), \quad  x \in \Omega, 
\]
belongs to $ L^{p'}(\Omega) $ and $ \|\rho\|_{p'} \le r $.

\end{enumerate}

Then, the problem $ \Psi(u) \in F(x, \Phi(u)) $ has at least one solution $ u \in U $ satisfying $ |\Psi(u)(x)| \le \rho(x) $ for almost every $ x \in \Omega $. 

\end{theorem}

Extending \cite[Theorem 3.1]{marano4} to the case  $ p \ne 2 $, we obtain the following result. We denote by $ \pi_0 $ and $ \pi_1 $ the projections of $ \Omega \times \R $ on $ \Omega $ and $ \R $, respectively. 

\begin{theorem} 
\label{disc}
Let $ \mathcal{F} = \{ A \subset \Omega \times \R: A \, \text{ is measurable and there exists} \, \, \,  i \in \{0,1\} \text{ such} \\ \text{that } m(\pi_i(A))=0\} $, let $(\alpha, \beta) \subset \R $ be an interval which does not contain $ 0 $,  let $\psi\colon (\alpha, \beta) \to \R $ be continuous, let $\varphi\colon \Omega \times \R \to \R$, and let $p>N$. Suppose that

\begin{enumerate}

\item[$(i)$] $\varphi$ is $ \mathcal L (\Omega \times \R)$-measurable and essentially bounded;

\item[$(ii)$] the set $D_{\varphi} = \{ (x,z) \in \Omega \times \R: \varphi$ is discontinuous at $(x,z) \}$  belongs to $\mathcal{F} $;

\item[$(iii)$] $\varphi^{-1}(r) \setminus \operatorname{int} (\varphi^{-1}(r)) \in \mathcal{F}$ for every  $r \in \psi((\alpha, \beta))$;

\item[$(iv)$] $\overline{\varphi (S \setminus D_{\varphi})} \subseteq \psi((\alpha, \beta))$.

\end{enumerate}

Then, there exists $u \in W^{1,p}_0(\Omega) $ such that
\[
\psi(-\Delta_p u) = \varphi (x,u) \quad \text{in } \Omega.
\]

\end{theorem}

We finally point out that existence results for implicit equations involving such operators have been obtained with very different techniques by \cite{AZA, CH, HS, SHT}.

\subsection{Structure of the paper} In Section 2 we will introduce the functional analytic setting we will use throughout the work. In Section 3 we will suppose $ \varphi(x, \cdot \,, \cdot) $ to be continuous. Here we will consider some cases, according to the growth conditions on $ \varphi $ or  to the choice of the set $ Y $. We  will also give  examples where these situations apply.  In Section 4 we will consider the discontinuous framework.

\section{Preliminaries}

Let $ X $ be a topological space and let $ V \subset X $. We denote by $ \operatorname{int}(V) $ the interior of $ V $ and by $ \overline V $ the closure of $ V $. The symbol $ \mathcal B(X) $ is used to denote the Borel $ \sigma$-algebra of $ X $.

If $ (X, d) $ is a metric space, for every $ x \in X, r \ge 0 $ and every nonempty set $ V \subset X $, we define
\[
B(x, r)= \{z \in X: d(x, z) \le r\} \quad \text{and} \quad d(x, V)= \inf_{z \in V} d(x, z).
\]
Let $ X $ and $ Z $ be two nonempty sets. A multifunction $ \Phi $ from $ X $ into $ Z $ (symbolically $ \Phi \colon X \to 2^Z $) is a function from $ X $ into the family of all subsets of $ Z $. A function $ \varphi\colon X \to Z $ is said to be a selection of $ \Phi $ if $ \varphi(x) \in \Phi(x) $ for all $ x \in X $. For every set $ W \subset Z $ we define $ \Phi^-(W)= \{x \in X: \Phi(x) \cap W \ne \emptyset \} $. 

Suppose that $ (X, \mathcal A) $ is a measurable space and $ Z $ is a topological space. We say that the multifunction $ \Phi $ is measurable if for every open set $ W \subset Z $ we have $ \Phi^-(W) \in \mathcal A $. Suppose now that  $ X $ and $ Z $ are two topological spaces. We say that  $ \Phi $ is lower semicontinuous (resp. upper semicontinuous) if  for every open (resp. closed) set $ W \subset Z $ the set $ \Phi^-(W) $ is open (resp. closed) in $ X $. When $ (Z, \delta) $ is a metric space, the multifunction $ \Phi $ is lower semicontinuous if and only if, for every $ z \in Z $, the real-valued function $ x \mapsto \delta(z, \Phi(x)) $, $ x \in X $, is upper semicontinuous (see \cite[Theorem 1.1]{ricceri3}). If, moreover, $ X $ is first countable, then  $ \Phi $ is lower semicontinuous if and only if, for every $ x \in X $, every sequence $ \{x_k\} $ in $ X $ converging to $ x $ and every $ z \in \Phi(x) $, there exists a sequence $ \{z_k\} $ in $ Z $ converging to $ z $ and such that $ z_k \in \Phi(x_k) $, for all $ k \in \N $ (see \cite[Theorem 7.1.7]{klein}).

A general result on the lower semicontinuity of a multifunction is the following

\begin{theorem}[Theorem 1.1 of \cite{ricceri1}]
\label{sci}
Let $ C, D $ be two topological spaces, with $ D $ connected and locally connected, and let $ f\colon C \times D \to \R $. For all $ x \in C $ we set 
\[
\begin{split}
& V(x):= \{y \in D : f(x, y)= 0\},  \\
& M(x):= \{y \in D: y \, \, \text{is a local extremum point for} \, \, f(x, \cdot)\}, \\
 \text{and} \quad & Q(x):= V(x) \setminus M(x).
\end{split}
\] 

Suppose that 
\begin{enumerate}

\item[(a)] for all $ x \in C, f(x, \cdot) $ is continuous, and $ 0 \in \operatorname{int}(f(x, D)) $;

\item[(b)] for all $ x \in C $ and for all $ A $ open subset of $ D $, there exists $ \bar y \in A $ such that $ f(x, \bar y) \ne 0 $;

\item[(c)] the set  
	\[
	\{(y', y'') \in D \times D : \{x \in C: f(x, y')< 0< f(x, y'')\} \, \, \text{is open} \}
	\]
 is dense in $ D \times D $.

\end{enumerate}

Then, the multifunction $ Q $ is lower semicontinuous, with nonempty closed values. 

\end{theorem}







From now on, $ \Omega $ is a bounded domain in $ \R^N $, $ N \ge 2 $, with a smooth boundary $ \partial \Omega $. 
The symbol $ \mathcal L(\Omega) $  denotes the Lebesgue $ \sigma$-algebra of $ \Omega $,  while $ m(\Omega) $ stands for the measure of $\Omega$.

Let $1 \le r< \infty$. We denote by $L^r(\Omega)$, $L^r(\Omega, \R^N)$, and $W^{1,r}(\Omega)$ the usual Lebesgue and Sobolev spaces equipped with the norms $\|\cdot\|_r$ and $\|\cdot\|_{1,r}$ given by
	\[
	\begin{split}
	\|u\|_r&= \l(\into |u|^r dx \r)^{1/r}, \quad \|\nabla u \|_r= \l(\into |\nabla u|^r dx \r)^{1/r}, \\
	\|u\|_{1, r}&= \l(\into |u|^r dx \r)^{1/r}+ \l(\into |\nabla u|^r dx \r)^{1/r}.
	\end{split}
	\]
For $r= \infty $ we recall that the norm of $L^{\infty}(\Omega)$ is given by
	\[
	\|u\|_{\infty}= \esssup_{\Omega} |u|.
	\]
Furthermore, we denote by $ W^{1,p}_0(\Omega) $ the closure of $ C^{\infty}_0(\Omega) $ in $ W^{1, p}(\Omega) $ and endow it with the norm 
\[
\|u\|:= \biggl(\int_{\Omega} |\nabla u(x)|^p dx \biggr)^{1/p}, \qquad u \in W^{1,p}_0(\Omega).\
\]
It is well known that the Sobolev embedding theorem guarantees the existence of a linear, continuous map $i\colon W^{1,p}_0(\Omega) \to L^{p^*}(\Omega)$, with the critical exponent given by
\[
p^*=
\begin{cases}
\frac{Np}{N-p} \quad & \, \text{if} \, \, \, p< N, \\
+\infty & \text{otherwise}.
\end{cases}
\]
In particular, the embedding $ W^{1,p}_0(\Omega) \hookrightarrow L^r(\Omega) $ is compact provided $ 1 \le r< p^* $.

If $ p \ne N $, then to each $ r \in [1, p^*] $ there corresponds a constant $ c_{rp}> 0 $ satisfying
\[
\|u\|_{r} \le c_{rp} \|u\|, \quad \forall \, u \in W^{1,p}_0(\Omega).
\]
On the other hand, when $ p= N $, for every $ r \in [1, \infty) $ we have
\[
\|u\|_{r} \le c_{rN} \|u\|, \quad \forall \, u \in W^{1,N}_0(\Omega).
\]
 When $ p> N $, the embedding  $W_0^{1,p}(\Omega) \hookrightarrow L^{\infty}(\Omega)$ implies the existence of a suitable $a>0$ such that
\begin{equation} 
\label{1dis}
\Vert u\Vert_{\infty} \leq a\Vert u\Vert, \quad \forall \, u \in W_0^{1,p}(\Omega),
\end{equation}
see \cite[Ch. IX]{brezis}. 

Given $ p \in  (1, \infty) $, the symbol $ p' $  denotes the conjugate exponent of $ p $ while $ W^{-1,p'}(\Omega) $ stands for the dual space of $ W^{1,p}(\Omega) $, with corresponding norm $\|\cdot\|_{-1, p'}$.  From \cite[Theorem 6.4]{brezis} we have the compact embedding  $ L^{p'}(\Omega) \hookrightarrow W^{-1,p'}(\Omega) $, and therefore there exists $ b> 0 $ such that
\begin{equation}
\label{2dis}
\|v\|_{{-1, p'}} \le b \|v\|_{p'}, \quad \forall \, v \in L^{p'}(\Omega).
\end{equation}
Let $ A_p\colon W^{1,p}_0(\Omega) \to W^{-1, p'}(\Omega) $ be the nonlinear operator stemming from the negative $p$-Laplacian, that is
\begin{equation}
\label{operator}
\langle A_p(u), v \rangle := \int_{\Omega} |\nabla u(x)|^{p-2} \nabla u(x) \cdot \nabla v(x) dx, \qquad u,v \in W^{1,p}_0(\Omega),
\end{equation}
and let $ \lambda_{1,p} $ be its first eigenvalue in $ W^{1,p}_0(\Omega) $. The following facts are well known  (see, e.g., \cite[Appendix A]{peral} or \cite{le}):

\begin{enumerate}

\item[(p$_1$)] $ A_p $ is bijective and uniformly continuous on bounded sets;

\item[(p$_2$)] the inverse operator $ A_p^{-1} $ is $ (W^{-1, p'}(\Omega), W^{1,p}_0(\Omega))$-continuous;

\item[(p$_3$)] $ \|A_p(u)\|_{-1,p'}= \|u\|_p^{p-1} $ in $ W^{1,p}_0(\Omega) $;

\item[(p$_4$)] $\displaystyle  \|u\|_p^p \le \frac{1}{\lambda_{1,p}} \|u\|^p $, for all $ u \in W^{1,p}_0(\Omega) $. 

\end{enumerate}

\section{The case when $ \varphi $ is a Carath\'eodory function}

In this section we consider the following problem: find $u \in W^{1, p}_0(\Omega)$ such that $\Delta_p u \in L^{p'}(\Omega)$ and
\begin{equation}
\label{equa}
\psi(-\Delta_p u)= \varphi(x, u, \nabla u). 
\end{equation}
We first suppose that $ Y= \R $ and state the following assumptions 

\begin{enumerate}

\item[(i)] $\psi$ is non-constant on intervals; 

\item[(ii)] for all $ (x, z, w)  \in \Omega \times \R \times \R^N $, the function $ y \mapsto \varphi(x, z,w)- \psi(y) $ changes sign.

\end{enumerate}



\begin{theorem}
\label{uno} 

Let $ \varphi \colon \Omega \times \R \times \R^N \to \R $ be a Carath\'eodory function and let $ \psi \colon \R \to \R $ be continuous. Suppose that (i)-(ii) hold true and, moreover,
\begin{enumerate}



\item[(iii)] there exist $ a \in L^{p'}(\Omega, \R^+_0), b, c \ge 0 $, with $\displaystyle  \frac{b}{\lambda_{1,p}}+ \frac{c}{\lambda_{1,p}^{1/p}}< 1 $, such that
\[
\sup \{\vert y \vert: y \in \psi^{-1}(\varphi(x, z, w)) \} < a(x)+ b|z|^{p-1}+ c|w|^{p-1}, 
\]
for all $ (x, z, w) \in \Omega \times \R \times \R^N $.

\end{enumerate}

Then, there exists a solution $ u \in W^{1,p}_0(\Omega) $ to equation \eqref{equa}.

\end{theorem}

\begin{proof}


Fix  $ x \in \Omega $. We want to apply Theorem \ref{sci}. To this end, we choose $ C= \R \times \R^N $, $ D= \R $, $ f(z,w,y)= \varphi(x, z, w)- \psi(y) $, and for every $ (z, w) \in \R \times \R^N $ we set
\[
\begin{split}
F(x, z, w):= \{y \in \R : \varphi(x, z, & w)- \psi(y)= 0,  \, \, \\
& \text{$ y $ is not a local extremum point of} \, \, \psi(\cdot)\}.
\end{split}
\]
Hypothesis (ii) directly yields (a). Moreover, in order to verify (b), we need to check that for all $ (z, w) \in \R \times \R^N $ the set $ U:= \{y \in \R: \varphi(x, z, w)- \psi(y) \ne 0\} $ is dense in $ \R $. Assumption (i) implies that $ \R \setminus U $ has empty interior, therefore $ U $ is dense in $ \R $, as desired. 

Let us next consider the set
	\begin{equation*}
	\begin{split}
	A:= \bigl\{(y',y'') \in \R \times \R : \{(z, w) \in \R \times \R^N : \varphi(x, &z, w)- \psi(y')< 0 \\
	&< \varphi(x, z, w)- \psi(y'')\} \, \, \text{is open}\bigr\}.
	\end{split}
	\end{equation*}
We want to show that $A$ is dense in $\R \times \R$. Suppose that there exist  $ y', y'' \in \R $ such that 
	\begin{equation}
	\label{y-primo}
	\varphi(x, z, w)- \psi(y')< 0< \varphi(x, z, w)- \psi(y''), 
	\end{equation}
that is,  $ \varphi(x, z, w) \in (\psi(y''), \psi(y')) $. Then the continuity of the function $ \varphi(x, \cdot \,, \cdot) $ implies that the set
	\begin{equation*}
	B:= \{(z, w) \in \R \times \R^N : \varphi(x, z, w)- \psi(y')< 0< \varphi(x, z, w)- \psi(y'')\}, 
	\end{equation*}
is open. If it is not possible to find such $y', y''$ that realize \eqref{y-primo}, then the set $B$ is empty. 
This implies that  $A= \R \times \R $, and then (c) follows.

Thanks to Theorem \ref{sci}, the multifunction $ F(x, \cdot \,, \cdot) $ is lower semicontinuous, with nonempty closed values. 




Moreover, thanks to  \cite[Lemma III.14]{castaing}, for all $ y', y'' \in \R $ we have
\begin{equation}
\label{measurable}
\begin{split}
&\{(x, z, w) \in \Omega \times \R \times  \R^N: \varphi(x, z, w)- \psi(y')< 0< \varphi(x, z, w)- \psi(y'')\} \\
=& \, \{(x, z, w) \in \Omega \times \R \times \R^N: \varphi(x, z, w) \in (\psi(y''), \psi(y')) \} \\
\in & \,   \mathcal L(\Omega) \otimes \mathcal B(\R \times \R^N).
\end{split}
\end{equation}
Therefore, setting $ \Lambda^*= \R \times \R $ we see that condition (iii) of \cite[Theorem 3.2]{marano1}  is satisfied. Fix now an open set $ A \subset \R  $. Arguing again as in \cite[Theorem 3.2]{marano1} we see that
	\[
	\begin{split}
	F^-(A)= \bigcup_{(y', y'') \in A \times A} \bigl\{(x, z, w) & \in \Omega \times \R \times \R^N: \\
	& \varphi(x, z, w)- \psi(y')< 0< \varphi(x, z, w)- \psi(y'')\bigr\}.
	\end{split}
	\]
Then \eqref{measurable} implies that $ F^-(A) \in \mathcal L(\Omega) \otimes \mathcal B(\R \times \R^N)$ and therefore  $ F $ is measurable.

Finally, fix any $ y \in F(x, z, w) $. By hypothesis (iii) we have
	\[
	\inf_{y \in F(x, z, w)} |y|< a(x)+ b|z|^{p-1}+ c|w|^{p-1} \quad \text{in} \, \, \, \Omega \times \R \times \R^N.
	\]
Therefore, all the hypotheses of Theorem \ref{mosconi} are satisfied, and there exists  $ u \in W^{1,p}_0(\Omega) $ such that $-\Delta_p u= F(x, u, \nabla u)$. By definition of $F$ we  then have  the result.
\end{proof}

\begin{remark}

We now discuss a very simple situation when hypothesis (iii) applies.

Suppose that  $ \varphi(\Omega \times \R \times \R^N) \subset [\alpha, \beta] $ and $ \psi $ is such that $ \psi^{-1}(B) $ is bounded, for every bounded $ B \subset \R $. If $ (x, z, w) \in \Omega \times \R \times \R^N $, we get $ \varphi(x, z, w) \in [\alpha, \beta] $, and so $ \psi^{-1}(\varphi(x, z, w)) \subset \psi^{-1}([\alpha, \beta]) $. Then, if we choose $ a \in L^{p'}(\Omega, \R^+_0) $ such that $ a(x) > \sup \{|y|: y \in \psi^{-1}([\alpha, \beta])\} $ for all $ x \in \Omega $, we have
\[
|\psi^{-1}(\varphi(x,z,w))| < a(x)  \quad \text{in }  \Omega \times \R \times \R^N,
\]
that is hypothesis (iii) with $b= c= 0$.

\end{remark}

As an application of the previous result, we consider the following 

\begin{corollary}
\label{ese-uno}

Let $ g \in L^{2}(\Omega) $ and $ \gamma \in (0, 1) $. Then, for every $ \lambda \ne 0$ and $  \mu \in \R $ there exists a solution $ u \in W^{1,2}_0(\Omega) $ to the equation
\begin{equation}
\label{eqex1}
-\Delta u= g(x) + \mu(|u|+ |\nabla u|)^{\gamma}+ \lambda \sin(-\Delta u).
\end{equation}

\end{corollary}

\begin{proof}

Fix $ \lambda \ne 0 $ and $ \mu \in \R $.  For every $ (x,z,w) \in \Omega \times \R \times \R^N  $ and every $ y \in \R $ we set
\[
\varphi(x,z,w):= g(x)+ \mu(|z|+ |w|)^{\gamma}  \quad \text{as well as} \quad \psi(y):= y- \lambda \sin y.
\]
Since $ \lim_{y \to \pm \infty} (y- \lambda \sin y)= \pm \infty $,  the function $ y \mapsto \varphi(x, z, w)- \psi(y) $  changes sign, and then hypothesis (ii) follows. Moreover, $\psi$ vanishes only at points of $ \R $ and not in intervals,  which implies that also hypothesis (i) is satisfied. 

Fix now $ (x, z, w) \in \Omega \times \R \times \R^N $. In order to verify hypothesis (iii), we want to find  $ b, c\ge 0 $, with $ \displaystyle \frac{b}{\lambda_{1,2}}+ \frac{c}{\lambda_{1,2}^{1/2}}< 1 $, and $ a \in L^2(\Omega, \R^+_0) $ such that
\begin{equation}
\label{max}
\max \bigl\{\vert y \vert: y \in \psi^{-1}(\varphi(x, z, w)) \bigr\}< a(x)+ b\vert z \vert+ c\vert w \vert,
\end{equation}
or equivalently $ \vert y \vert< a(x)+ b\vert z \vert+ c\vert w \vert $ for every $ y $ solution  to the equation 
	\begin{equation}
	\label{eqesempio}
	\psi(y)= \varphi(x, z, w).
	\end{equation}
We point out that in \eqref{max} the maximum replaces the supremum because the set  $\psi^{-1}(\varphi(x, z, w))$ is compact. Let $\tilde y$ be a solution to \eqref{eqesempio}. Then Young's inequality with exponents $ 1/{\gamma} $ and $ 1/(1- \gamma) $ gives
	\begin{equation}
	\label{young}
	\begin{split}
	|\psi(\tilde y)|= |\varphi(x, z,w)| &= |g(x)+ \mu(|z|+ |w|)^{\gamma}| \\
	&\le |g(x)|+  |\mu| |z|^{\gamma}+  |\mu| |w|^{\gamma} \\
	& \le |g(x)|+ \eps |z|+ \eps |w|+ C_{\gamma, \eps, \mu} \\
	& \le \tilde g(x)+ \eps |z|+ \eps |w|,
	\end{split}
	\end{equation}
where $ \tilde g(x):= |g(x)| + C_{\gamma, \eps, \mu} $ for every $ x \in \Omega $. 
On the other hand, by the definition of $ \psi $ we have
\[
\vert \psi(\tilde y) \vert= \vert \tilde y- \lambda \sin \tilde y \vert \ge \vert \tilde y \vert- \vert \lambda \vert,
\]
and then \eqref{young} gives 
\[
\begin{split}
\vert \tilde y \vert & \le \vert \psi(\tilde y) \vert+ \vert \lambda \vert \\
&\le \tilde g(x)+ \vert \lambda \vert+ \eps |z|+ \eps |w| \\
& < \bar g(x)+ \eps \vert z \vert+ \eps \vert w \vert,
\end{split}
\]
where $ \bar g(x):=  \tilde g(x) + 2 \vert \lambda \vert $, for every $ x \in \Omega $. Observe that $ \bar g \in L^2(\Omega, \R^+_0) $. If we choose $ \eps $ in such a way that
\[
\frac{\eps}{\lambda_{1,2}}+ \frac{\eps}{\lambda_{1,2}^{1/2}}< 1,
\]
then hypothesis (iii) is satisfied  with $ a:= \bar g $ and $ b:= c:= \eps $. Thanks to Theorem \ref{uno}, there exists a solution $ u \in W^{1,2}_0(\Omega) $ to  equation \eqref{eqex1}.

\end{proof}

In the following situation the function $ \psi $ exhibits a very different behavior.

\begin{corollary}

Let $ p \in [2, +\infty) $, $ f \in L^{p'}(\Omega) $, and  $ \gamma \in  (0, p-1) $. Then, for every $ \mu \in \R $ and $ \lambda> 0 $, there exists a solution $ u \in W^{1,p}_0(\Omega) $ to the equation
\begin{equation}
\label{eq}
-\Delta_p u= f(x) + \mu (|u|+ |\nabla u|)^{\gamma}- \lambda e^{-\Delta_p u}.
\end{equation}

\end{corollary}

\begin{proof}

Fix $ \mu \in \R $ and $ \lambda>0 $. As before, for every $ (x, z, w) \in \Omega \times \R \times \R^N $ and every $ y \in \R $ we set
\[
\varphi(x, z, w):= f(x)+ \mu (|z|+ |w|)^{\gamma}  \quad \text{as well as} \quad \psi(y):= y+ \lambda e^y.
\]
Since $ \lim_{y \to \pm \infty} (y+ \lambda e^y)= \pm \infty $, then hypotheses (i) and (ii)  are fulfilled. In order to verify hypothesis (iii), we argue as in Corollary \ref{ese-uno}.  Let $ \tilde y $ be a solution to $ \varphi(x, z, w)- \psi(y)= 0 $, then  Young's inequality  with exponents $ \displaystyle  \frac{p-1}{\gamma}, \frac{p-1}{p-1- \gamma} > 1 $ gives
	\[
	\begin{split}
	|\psi(\tilde y)|= |\varphi(x, z, w)| &= |f(x)+ \mu (|z|+ |w|)^{\gamma}| \\
	& \le |f(x)|+ 2^{\gamma}(|\mu| |z|^{\gamma}+ |\mu| |w|^{\gamma}) \\
	& \le |f(x)|+ \eps |z|^{p-1}+ \eps  |w|^{p-1}+ C_{\gamma, \eps, \mu} \\
	&= \tilde f(x)+ \eps |z|^{p-1}+ \eps  |w|^{p-1}, 
	\end{split}
	\]
where $ \tilde f(x):= |f(x)| + C_{\gamma, \eps, \mu} $ for every $ x \in \Omega $.

On the other hand we have
	\begin{equation}
	\label{xi}
	\vert \psi(\tilde y) \vert= \vert \tilde y+ \lambda e^{\tilde y} \vert \ge \vert \tilde y \vert- \vert \xi \vert,
	\end{equation}
$ \xi \not\equiv 0$ being the unique solution to the equation $ y+ \lambda e^{y}= 0 $. Let us show \eqref{xi} for a general $ y \in \R $. If $ y \ge \xi $ we have
	\[
	\begin{split}
	\vert y+ \lambda e^y \vert &= \vert y+ \lambda e^y- \xi- \lambda e^{\xi} \vert \\
	&= \vert y- \xi+ \lambda(e^y- e^{\xi})\vert \\
	& \ge \vert y- \xi \vert \ge \vert y \vert- \vert \xi \vert.
	\end{split}
	\]
Suppose now that $ y< \xi $, then
\[
\begin{split}
\vert y+ \lambda e^y \vert&= \vert y-	\xi+ \lambda (e^y- e^{\xi}) \vert \\
&= \vert \xi- y+  \lambda(e^{\xi}- e^y) \vert \\
& \ge \vert \xi- y \vert \ge \vert y \vert- \vert \xi \vert.
\end{split}
\]
From \eqref{xi} we then have
\[
\begin{split}
\vert \tilde y \vert & \le \vert \psi(\tilde y) \vert+ \vert \xi \vert \\
&\le \tilde f(x)+ \eps |z|^{p-1}+ \eps  |w|^{p-1} + \vert \xi \vert \\
& < \bar f(x)+ \eps |z|^{p-1}+ \eps  |w|^{p-1},
\end{split}
\]
with $ \bar f(x):=  \tilde f(x) + 2 \vert \xi \vert $ for every $ x \in \Omega $. Observe that $ \bar f \in L^{p'}(\Omega, \R^+_0) $. Then, if we choose $ \eps $ in such a way that
\[
\frac{\eps}{\lambda_{1,p}}+ \frac{\eps}{\lambda_{1,p}^{1/p}}< 1,
\]
hypothesis (iii) is satisfied with $ a:= \bar f $ and $ b:= c:= \eps $. Therefore, Theorem \ref{uno} gives the existence of  a solution $ u \in W^{1,p}_0(\Omega) $ to equation \eqref{eq}.
\end{proof}

In order to state our next theorem, we need some preliminary results. The following is an a priori estimate on $\|\nabla u\|_{L^{\infty}(\Omega; \R^N)}$, see \cite[Proposition 3.3]{marmos} or \cite[Theorem 1.3]{cianchi}.

\begin{proposition} \label{prop-cianchi}

Suppose $q> N$. Then, there exists a constant $\hat C>0$, depending on $p, q$, and $\Om$, such that
	\[
	\|\nabla u\|_{L^{\infty}(\Omega, \R^N)} \le \hat C \|\Delta_p u\|_{L^q(\Omega)}^{1/(p-1)}.
	\]
\end{proposition}

Proposition \ref{prop-cianchi} is used in the proof of the following

\begin{theorem}
\label{thm-mosconi}

Let $p \in (1, \infty)$, $q>N$, and let $F\colon \Om \times \R \times \RN \to 2^{\R}$ be a closed-valued multifunction. Suppose that

\begin{itemize}

\item[(h$_1$)] $F$ is $\mathcal L(\Om) \otimes \mathcal B(\R \times \RN)$-measurable;

\item[(h$_2$)] for almost every $x \in \Om$ the multifunction $(z, w) \mapsto F(x, z, w)$ turns out to be lower semicontinuous;

\item[(h$_3$)] for appropriate $a \in L^q(\Om, \R_0^+)$ and $\xi\colon \R_0^+ \times \R_0^+ \to  \R_0^+$ nondecreasing with respect to each variable separately one has
	\[
	\inf_{y \in F(x, z, w)} |y|< a(x)+ \xi(|z|, |w|) \quad \text{in } \Om \times \R \times \RN;
	\]

\item[(h$_4$)] there exists $R>0$ such that
	\[
	\|a\|_q+ m(\Om)^{1/q} \xi(\delta_{\Om} \hat C R^{1/(p-1)}, \hat C R^{1/(p-1)}) \le R,
	\]
where $\delta_{\Om}:= \operatorname{diam}(\Om)$ and $\hat C$ is given by Proposition \ref{prop-cianchi}. 

\end{itemize}

Then, there exists at least one solution $u \in W^{1,p}_0(\Om)$ to problem
	\[
	\begin{aligned}
	-\Delta_p u& \in F(x, u, \nabla u) \quad && \text{in } \Om, \\
	u&= 0 && \text{on } \rand.
	\end{aligned}
	\]

\end{theorem}

Finally, we state our result. 

\begin{theorem}
\label{due}

Let $ \varphi $ and $ \psi $ as in Theorem \ref{uno}. Suppose that hypotheses (i)-(ii) hold true and, moreover,
\begin{enumerate}
\item[(iii)$'$]  there exist $ a \in L^q(\Omega, \R^+_0) $, $ q> N $, $ g\colon \R^+_0 \times \R^+_0 \to \R^+_0 $ nondecreasing with respect to each variable separately, such that
\[
\sup \{\vert y \vert: y \in \psi^{-1}(\varphi(x, z, w)) \} < a(x)+ g(|z|, |w|), 
\]
for all $ (x, z, w) \in \Omega \times \R \times \R^N $;

\item[(iv)] there exists $ R> 0 $ such that
\[
\|a\|_{L^q(\Omega)}+ m(\Omega)^{1/q} g(\delta_{\Omega} \hat C R^{1/(p-1)}, \hat C R^{1/(p-1)}) \le R,
\]
where  $\hat C$ comes from Proposition \ref{prop-cianchi}.
\end{enumerate}
Then, equation \eqref{equa} has a solution $ u \in W^{1,p}_0(\Omega) $.

\end{theorem}

\begin{proof}

We aim to apply Theorem \ref{thm-mosconi}. As before, fix $ x \in \Omega $ and for all $ (z, w) \in \R \times \R^N $ define 
\[
\begin{split}
F(x, z, w):= \{y \in \R : & \, \, \varphi(x, z,w)-  \psi(y)= 0, \\
& y \, \, \text{is not a local extremum point of} \, \ \psi(\cdot)\}.
\end{split}
\]
Reasoning as in Theorem \ref{uno} ensures that $ F $ has nonempty  closed values, is lower semicontinuous w.r.t. $ (z, w) $, and $ \mathcal L(\Omega) \otimes \mathcal B(\R \times \R^N) $-measurable. 

Fix now  $ y \in F(x, z, w) $, that is  $ y \in \psi^{-1}(\varphi(x, z, w)) $. Then hypothesis (iii)$'$ implies that
\[
\inf_{y \in F(x, z, w)} |y|< a(x)+ g(|z|, |w|) \quad \text{in } \Omega \times \R \times \R^N.
\]
Taking into account (iv), we see that all the hypotheses of Theorem \ref{thm-mosconi} are fulfilled. Therefore, there exists $ u \in W^{1,p}_0(\Omega) $ such that $ -\Delta_p u \in F(x, u, \nabla u) $. According to the definition of $ F $, it turns out that $ u $ is a solution to equation \eqref{equa}.
\end{proof}

The following result is an application of the previous theorem and has been inspired by \cite[Corollary 1]{kacz}. Observe that, unlike \cite{kacz}, here we consider a function $ \varphi $ which is not necessarily continuous w.r.t. the variable $ x $, but only lies in a suitable $ L^q(\Omega) $. Moreover, here we deal with partial differential equations. 


\begin{corollary}

Let $ h \in L^q(\Omega) $, with $ q> N $. Then, for every $ k \ne 0 $ and every sufficiently small $ \Vert h \Vert_q $ there exists a solution $ u \in W^{1,2}_0(\Omega) $ to the equation
\begin{equation*}
\label{laplacian}
-\Delta u= h(x) + u^3+ \vert \nabla u \vert^2+ k \sin (-\Delta u).
\end{equation*}

\end{corollary}

\begin{proof}

Fix $ k \in \R $ and for all $ (x, z, w) \in \Omega \times \R \times \R^N$ and all $ y \in \R $ define 
\[
\varphi(x, z, w):= h(x)+ z^3+ \vert w \vert^2 \quad \text{as well as} \quad \psi(y):= y- k \sin y. 
\]
Reasoning like in Corollary \ref{ese-uno} gives that hypotheses (i)-(ii) are fulfilled. 

In order to verify hypothesis (iii)$'$, let $ g(|z|, |w|):= |z|^3+ |w|^2 $ for all $ (z, w) \in \R \times \R^N $.  It turns out that $  g\colon \R^+_0 \times \R^+_0 \to \R^+_0  $ is nondecreasing w.r.t. each variable, separately. Let  $ \tilde y $ be a solution to the equation $ \psi(y)= \varphi(x, z, w) $. It follows that
\begin{equation*}
\label{eseh}
\begin{split}
\vert \psi(\tilde y) \vert &= |\varphi(x, z, w)| \\
& \le |h(x)|+ |z|^3+ |w|^2 \\
&= \vert h(x) \vert+ g(\vert z \vert, \vert w \vert).
\end{split}
\end{equation*}
On the other hand, since  $ \vert \psi(\tilde y) \vert= \vert \tilde y- k \sin \tilde y \vert \ge \vert \tilde y \vert- \vert k \vert $, then we have
\[
\begin{split}
\vert \tilde y \vert & \le \vert \psi(\tilde y) \vert+ \vert k \vert \\
& \le \vert h(x) \vert+ g(\vert z \vert, \vert w \vert)+ \vert k \vert \\
& < \bar h(x)+ g(\vert z \vert, \vert w \vert),
\end{split}
\]
where $ \bar h(x):= |h(x)|+ 2 \vert k \vert $ for every $ x \in \Omega $ and $ \bar h \in L^q(\Omega, \R^+_0) $.  Hence  hypothesis (iii)$'$ follows.

In order to verify hypothesis (iv), we have to check the existence of $ R> 0 $ such that
\begin{equation}
\label{diseg}
\|\bar h\|_{L^q(\Omega)}+ m(\Omega)^{1/q} \delta_{\Omega}^3 \hat C^3 R^3+ m(\Omega)^{1/q} \hat C^2 R^2 \le R.
\end{equation}
If $ 0< R <\!<1 $, then choosing $ \bar h $ in such a way that $ \| \bar h\|_{L^q(\Omega)}< \frac{R}{2} $ gives  immediately \eqref{diseg}, since the terms containing $ R^2 $ and $ R^3 $ are negligible with respect to $ R $.  Therefore, all the hypotheses of Theorem \ref{due} are fulfilled, and we have the thesis. 
\end{proof}

The next result provides solutions to equation \eqref{equa} when the function $ \psi $ is of the form $ y \mapsto y- h(y)$, with $ h $ continuous and bounded. Note that here  a specific growth condition on $ \varphi $ is required.

\begin{theorem}

Let $ \varphi\colon \Omega \times \R \times \R^N \to \R $ be a Carath\'eodory function and let $ h \in L^{\infty}(\R) $ be continuous. Suppose that (i)-(ii) hold true and, moreover,
\begin{enumerate}

\item[(iii)$''$] there exist $ f \in L^{p'}(\Omega, \R^+_0)$, with $ f(x) \ge \|h\|_{\infty} $ for all $ x \in \Omega $, $ \mu> 0 $, and $ \gamma \in (0, p-1) $ such that
\[
\sup_{(x, z, w) \in \Omega \times \R \times \R^N} |\varphi(x, z, w)| < f(x)+ \mu(|z|+ |w|)^{\gamma}.
\]
\end{enumerate}
Then, there exists a solution $ u \in W^{1,p}_0(\Omega) $ to the equation
\begin{equation}
\label{h}
-\Delta_p u- h(-\Delta_p u)= \varphi(x, u, \nabla u).
\end{equation}

\end{theorem}

\begin{proof}

We fix $ x \in \Omega $ and for all $ (z, w) \in \R \times \R^N $ define 
\[
\begin{split}
F(x, z, w):= \{y \in \R : \, \, & \varphi(x, z,w)- (y- h(y))= 0, \\
& y \, \, \text{is not a local extremum point of} \,  \, y \mapsto y- h(y)\}.
\end{split}
\]
Reasoning as in the above proofs ensures that $ F $ is lower semicontinuous w.r.t. $ (z, w) $, $ \mathcal L(\Omega) \otimes \mathcal B(\R \times \R^N) $-measurable, and has nonempty,  closed values.

Fix  $ (x, z, w) \in \Omega \times \R \times \R^N $. If $ y \in F(x, z, w) $, then it solves the equation $ \varphi(x, z, w)= y- h(y) $. We first suppose that $ \gamma \in  [1, p-1) $. Then Young's inequality with exponents $ \frac{p-1}{\gamma}, \frac{p-1}{p-1-\gamma}> 1 $ gives
\[
\begin{split}
|y| & \le |y- h(y)|+ |h(y)| \\
&\le |\varphi(x, z, w)|+ \|h\|_{\infty} \\
& < f(x)+ \mu(|z|+ |w|)^{\gamma}+ \|h\|_{\infty} \\
& \le 2 f(x) +  2^{\gamma-1} \mu (|z|^{\gamma}+ |w|^{\gamma}) \\
& \le  2 f(x)+  2^{\gamma-1} \mu (\eps|z|^{p-1}+ \eps |w|^{p-1}+ K_{\eps})  \\
& \le 2 f(x)+ C_{\eps}+   2^{\gamma-1} \mu \eps(|z|^{p-1}+ |w|^{p-1}),
\end{split}
\]
where $ C_{\eps}:=   2^{\gamma-1} \mu K_{\eps} $. Hence
\[
\inf_{y \in F(x, z, w)} |y|< 2 f(x)+ C_{\eps}+ 2^{\gamma-1} \mu \eps(|z|^{p-1}+ |w|^{p-1}).
\]
If we choose $ \eps $ in such a way that
\[
\frac{2^{\gamma-1} \mu \eps}{\lambda_{1,p}}+ \frac{2^{\gamma-1} \mu \eps}{\lambda_{1,p}^{1/p}}< 1,
\]
hypothesis (h$3$) of Theorem \ref{mosconi} is fulfilled with $ a:= 2 f+ C_{\eps} \in L^{p'}(\Omega, \R^+_0) $ and $ b:= c:= 2^{\gamma-1} \mu  \eps $. 

Suppose now  $ \gamma \in (0, 1) $. Since $ (a+ b)^{\gamma} \le a^{\gamma}+ b^{\gamma} $ for every $ a, b \ge 0 $, reasoning as before yields 
\[
\vert y \vert < 2 f(x)+ \tilde C_{\eps}+ \mu \eps(\vert z \vert^{p-1}+ \vert w \vert^{p-1}),
\]
where $ \tilde C_{\eps}:= \mu K_{\eps} $. If we now choose $ \eps $ in such a way that 
\[
\frac{\mu \eps}{\lambda_{1,p}}+ \frac{\mu \eps}{\lambda_{1,p}^{1/p}}< 1,
\]
hypothesis (h$3$) of Theorem \ref{mosconi} is again fulfilled with $ a:= 2 f+ \tilde C_{\eps} \in L^{p'}(\Omega, \R^+_0) $ and $ b:= c:= \mu  \eps $. 

In both cases, there exists $ u \in W^{1,p}_0(\Omega)$ such that $ -\Delta_p u \in F(x, u, \nabla u) $, which gives a solution to equation \eqref{h}.
\end{proof}

We conclude this section considering the case when $ Y $ is a closed interval of $ \R $. Observe that here no growth conditions on $ \varphi $ are required. 

\begin{theorem}
\label{intervallo}

Let $ \varphi \colon \Omega \times \R \times \R^N \to \R $ be a Carath\'eodory function and let $ \psi\colon [\alpha, \beta] \to \R $ be continuous. Suppose that
\begin{enumerate}

\item  $\psi$ is non-constant on intervals;

\item for every $ (x, z, w)  \in \Omega \times \R \times \R^N $,  the function $ y \mapsto \varphi(x, z,w)- \psi(y) $ changes sign in $ [\alpha, \beta] $.

\end{enumerate}

Then, there exists a solution $ u \in W^{1,p}_0 (\Omega) $  to equation \eqref{equa}. 

\end{theorem}

\begin{proof}

As before, fix $ x \in \Omega $ and for all $ (z, w) \in \R \times \R^N $ define 
\[
\begin{split}
F(x, z, w):= \{y \in [\alpha, \beta] : & \, \,  \varphi(x, z,w)- \psi(y)= 0, \\
& y \, \, \text{is not a local extremum point of} \, \, \psi(\cdot)\}.
\end{split}
\]
A familiar argument ensures that   $ F $ takes nonempty  closed values, is lower semicontinuous w.r.t. $ (z, w)$ and $ \mathcal L(\Omega) \otimes \mathcal B(\R \times \R^N) $-measurable. 

If  now $ y \in F(x,z,w) $, then $ |y| \le \max \{ \vert \alpha \vert, \vert \beta \vert \} $, and so hypothesis (h$3$) of Theorem \ref{mosconi} is immediately satisfied with $ a(x):= 2 \max \{ \vert \alpha \vert, \vert \beta \vert \} $ for every $ x \in \Omega $  and $ b:= c:= 0 $.  Therefore, there exists $ u \in W^{1,p}_0 (\Omega) $  such that $ -\Delta_p u \in F(x, u, \nabla u ) $, i.e., $ u $ is a solution to  \eqref{equa}.
\end{proof}

We now consider two applications of the  previous result, which differ by the behavior of the function  $ \psi $. In both  cases, the  boundedness of $ \varphi $ will play a central role.

\begin{corollary}

Let $ f \in L^{\infty}(\Omega) $, $ k \in \N $, $ k $ even and such that $ k \pi> \Vert f \Vert_{\infty} $, and let $ \psi\colon [-k \pi, k \pi] \to \R $ be defined by $ \psi(y)= y \cos y $. Then, there exists a solution $ u \in W^{1,p}_0(\Omega) $ to the equation
\begin{equation}
\label{appl}
\psi(-\Delta_p u)= f(x) \quad \text{in }  \Omega.
\end{equation}

\end{corollary}

\begin{proof} 

Assumption (1) is clearly satisfied. Moreover, for every $ x \in \Omega $, we have
\[
\begin{split}
& f(x)- \psi(k \pi)= f(x)- k \pi \, \cos(k \pi)= f(x)- k \pi \, (-1)^k= f(x)- k \pi< 0 \\
 \text{and} \qquad & f(x)- \psi(-k \pi)= f(x)+ k \pi \cos(-k \pi)= f(x)+ k \pi> 0,
\end{split}
\]
which gives hypothesis (2). Thanks to Theorem \ref{intervallo}, there exists at least a solution $ u \in W^{1,p}_0(\Omega) $ to equation \eqref{appl}.
\end{proof}

Note that the interval $ [\alpha, \beta] $ could  be  unbounded, as  the following example shows.

\begin{corollary}

Let $ p \in (1, \infty) $, $ f \in L^{p'}(\Omega)$,   and $ \varphi\colon \Omega \times \R \times \R^N \to \R $. Suppose that there exists $ \lambda>0 $ such that
\begin{equation}
\label{iplambda}
\sup_{(x, z, w) \in \Omega \times \R \times \R^N} \vert \varphi(x, z, w) \vert < \lambda.
\end{equation}
Then, there exists a solution $ u \in W^{1, p}_0(\Omega) $ to the equation 
\begin{equation*}
\label{eqex}
\varphi(x, u, \nabla u)- \lambda e^{\Delta_p u}+ \Delta_p u= 0.
\end{equation*}

\end{corollary}

\begin{proof}

Define $ \psi(y):= \lambda e^{-y}- y $ for every $ y \in [0, +\infty) $. Observe that hypothesis (1) is immediately satisfied. Moreover, thanks to \eqref{iplambda}, for every $ (x, z, w) \in \Omega \times \R \times \R^N $ we have  
\[
\begin{split}
& \varphi(x, z, w)- \psi(0)= \varphi(x, z, w)- \lambda< 0 \\
\text{and} \quad & \lim_{y \to +\infty} (\varphi(x, z, w)- \psi(y))= +\infty,
\end{split}
\]
that is  hypothesis (2), and hence the conclusion follows from Theorem \ref{intervallo}.
\end{proof}

\section{The discontinuous framework}

This section is  devoted to the proof of Theorem \ref{disc}, which we rewrite here, for the reader's convenience. Given $(x,z)\in S:= \Omega \times \R  $, set $\pi_0(x,z)=x$ and $\pi_1(x,z)=z $. Moreover,  fix $ p> N $ and define
\[
\mathcal{F}= \{ A \subset S: A \, \, \text{is measurable and there exists } i \in \{0,1\} \text{ such that } m(\pi_i(A))=0\}.
\]

\begin{theorem} 

\label{ultimo}

Let $ (\alpha, \beta) \subset \R $ be such that $ 0 \notin (\alpha, \beta) $,  let $\psi\colon (\alpha, \beta) \to \R $ be  continuous, and $\varphi \colon \Omega \times \R \to \R $. Suppose that

\begin{enumerate}

\item[$(i)$] $\varphi$ is $ \mathcal L(\Omega \times \R)$-measurable and essentially bounded;

\item[$(ii)$] the set $D_{\varphi} = \{ (x,z) \in S: \varphi$ is discontinuous at $(x,z) \}$  belongs to $\mathcal{F} $;

\item[$(iii)$] $\varphi^{-1}(r) \ \setminus \operatorname{int} (\varphi^{-1}(r)) \in \mathcal{F}$ for every  $r \in \psi((\alpha, \beta))$;

\item[$(iv)$] $\overline{\varphi (S \setminus D_{\varphi})} \subset \psi((\alpha, \beta))$.

\end{enumerate}

Then, there exists $u \in W^{1,p}_0(\Omega)$ such  that
\begin{equation}
\label{eq-disc}
\psi(-\Delta_p u) = \varphi (x,u) \quad \text{in }  \Omega.
\end{equation}

\end{theorem}

\begin{proof}

The first part essentially follows the proof of  \cite[Theorem 3.1]{marano4}. Thanks to assumption (i), there exists a constant $c > 0$ such that
\[
S \setminus D_{\varphi} \subset \{ (x,z) \in S: \vert \varphi(x,z) \vert \leq c \}.
\]
Set 
\[
\hat a := \min{ \overline{ \varphi(S \backslash D_{\varphi}) }} \ \ \ \text{and} \ \ \
\hat b := \max{\overline{\varphi(S \backslash D_{\varphi})}} .
\]
Thanks to hypothesis (iv) there exist $y', y'' \in (\alpha, \beta) $  such that $\psi(y') = \hat a $ and $  \psi(y'') = \hat b$. Let $\lambda\colon [0,1] \rightarrow (\alpha, \beta) $ be  a continuous function such that $\lambda(0)=y'$, $\lambda(1) = y''$. Moreover, let $\tilde \psi \colon [0,1] \to \R$ be defined by
	\[
	\tilde \psi (t):= \psi(\lambda(t)), \quad t \in [0,1].
	\]
We distinguish among two cases. 

Suppose that $ \tilde \psi $ is constant. Then $\hat a= \hat b$ and  consequently  $\varphi(S \setminus D_{\varphi}) = \{ \hat a \}$. Let $ u \in W_0^{1,p} (\Omega) $ be such that $-\Delta_p u= y'$. Since $\psi( -\Delta_p u)=\psi(y')= \hat a$,  the conclusion will be achieved by showing that the set 
	\[
	\Omega_{\varphi}:= \{ x \in \Omega: (x, u(x)) \in D_{\varphi} \}
	\]
 has measure zero. 

First of all observe that an elementary computation gives 
\begin{equation} 
\label{Omegaphi}
\Omega_{\varphi} \subset \pi_0(D_{\varphi}) \cap u^{-1} (\pi_1(D_{\varphi}))
\end{equation}
and, due to (ii), $ m(\pi_i(D_{\varphi})) = 0 $ for some $i \in \{ 0,1\}$. Suppose $ i= 0 $. From \eqref{Omegaphi} we obtain
\[
m(\Omega_{\varphi}) \le m(\pi_0(D_{\varphi}) \cap u^{-1} (\pi_1(D_{\varphi}))) \le m(\pi_0(D_{\varphi}))= 0,
\]
which implies $ m(\Omega_{\varphi})= 0 $. Let now $ i= 1 $. From \cite[Lemma 1]{buttazzo} we have $ \nabla u(x)= 0 $ a.e. in $ u^{-1}(\pi_1(D_{\varphi})) $ which in other words is 
	\begin{equation}
	\label{u-meno1}
	u^{-1}(\pi_1(D_{\varphi})) \subset \{x \in \Omega: \, \nabla u(x)= 0\}.
	\end{equation}	
Thanks to \cite[Theorem 1.1]{lou}, we have  $ y'= 0 $ on $ \{x \in \Omega: \, \nabla u(x)= 0 \} $, which in particular holds on $ u^{-1}(\pi_1(D_{\varphi})) $, taking into account \eqref{u-meno1}. Since $ y' \in (\alpha, \beta) \not\ni 0 $, this is possible if and only if $ m(u^{-1}(\pi_1(D_{\varphi})))= 0 $. From \eqref{Omegaphi} we then have 
\[
m(\Omega_{\varphi}) \le m(\pi_0(D_{\varphi}) \cap u^{-1} (\pi_1(D_{\varphi}))) \le m(u^{-1} (\pi_1(D_{\varphi}))),
\]
which implies $ m(\Omega_{\varphi})= 0 $. Hence  the thesis follows

Suppose now that $ \tilde \psi $ is non constant and choose $t_1,t_2 \in [0,1]$ such that 
\[
\tilde \psi (t_1) = \min_{t \in [0,1]}  \tilde \psi (t) \quad \text{as well as} \quad \tilde \psi (t_2) = \max_{t \in [0,1]} \tilde \psi (t).
\]
Obviously, $t_1 \neq t_2$ and there is no loss of generality in assuming $t_1 < t_2$. Let $h\colon \tilde \psi ([0,1]) \to [0,1]$ be defined by 
	\[
	h(r) = \min{(\tilde \psi^{-1}(r) \cap [t_1,t_2])}, \quad \forall \, r \in \tilde \psi ([0,1]).
	\] 
We claim that $h$ is strictly increasing. Indeed, let $r_1, r_2 \in \tilde \psi ([0,1])$ be such that $r_1 < r_2$. Then, $h(r_1) \neq h(r_2)$ and $t_1 < h(r_2)$. Taking into account that $ \tilde \psi (h(r_2))=r_2 > r_1$, $ \tilde \psi (t_1) \le r_1$, and the continuity of $ \tilde \psi $, we immediately infer $h(r_1) < h(r_2)$. 

Therefore, the family $D_k$ of all discontinuity points of the function $k\colon \mathbb{R} \rightarrow (\alpha, \beta)$ given by
\[
k(r)=\begin{cases}
 \ \lambda (h(\tilde{\psi}(t_1))) \quad & \, \text{if} \, \, \, r \in  (-\infty, \tilde{\psi}(t_1)) \\
 \ \lambda (h(r)) \quad & \, \text{if} \, \, \, r \in \tilde{\psi}([0,1]) \\
 \ \lambda (h(\tilde{\psi}(t_2))) \quad & \, \text{if} \, \, \, r \in (\tilde{\psi}(t_2),+\infty)
\end{cases}
\]
is at most countable. Owing to hypotheses (ii) and (iii), this implies that the set

\begin{equation}
\label{D}
D = D_{\varphi} \cup \l\{ \bigcup_{r \in D_k} \bigl[\varphi^{-1}(r) \backslash \operatorname{int}(\varphi^{-1}(r)) \bigr] \r\}
\end{equation}
has measure zero.

Define now $f\colon S \to \R $ by $ f(x, z):= k(\varphi(x,z))$. Since $f(S) \subset \lambda([0,1])$ it follows that $f$ is bounded. Moreover, arguing as in \cite[Theorem 3.1]{marano4} gives that $ f $ is continuous.  Set now
\begin{equation*}
F(x,z) := \overline{\operatorname{co}} \left( \bigcap_{\delta > 0 } \bigcap_{ E \in \mathcal{E} } \overline{f(B_{\delta}(x,z) \setminus E)} \right),
\end{equation*}
where
\[
\begin{split}
\mathcal{E} &= \{ E \subset S : m(E)= 0 \} \\
\text{and} \qquad B_{\delta}(x,z) &= \{ (x',z') \in S: \vert x-x' \vert + \vert z - z' \vert \leq \delta \}.
\end{split}
\]
A standard argument (see, e.g,  \cite[Theorem 3.1]{marano4}), ensures that $F$ is upper semicontinuous, with nonempty, convex, and closed values. Furthermore, $ F(\cdot \,, z) $ is measurable for every $ z \in \R $, $ F(x, \cdot) $ has a closed graph for almost all $ x \in \Omega $, and it holds
\begin{equation*}
F(x,z) = \{ f(x,z) \} \quad \text{as soon as }  (x,z) \in S \setminus D.
\end{equation*}
Consider now the problem 
\begin{equation}
\label{prob}
-\Delta_p u \in F(x, u)  \, \,  \, \text{in} \, \, \Omega, \quad u \in W^{1,p}_0(\Omega).
\end{equation}
We want to show existence of solutions to \eqref{prob} by means of Theorem \ref{plap}. To this end, let us verify  hypotheses ($i_1$)-($i_4$). If $ A_p $ is the operator given in \eqref{operator}, we choose
		\[
		U:=A_p^{-1}(L^{p'}(\Omega)), \quad \Phi(u):=u \quad \text{and} \quad \Psi(u):=A_p(u),
		\]
 for every $u \in U$. Observe in particular that $ A_p\colon U \rightarrow L^{p'}(\Omega)$ is bijective. 

Let $v_h \rightharpoonup v$ in $L^{p'}(\Omega)$. Since $ \{v_h\} $ is bounded in $ L^{p'}(\Omega) $, and $ L^{p'}(\Omega) $ compactly embeds in $ W^{-1,p'}(\Omega)$, there exists a subsequence, still denoted by $ \{v_h\} $, such that $ v_h \to v $ in $ W^{-1,p'}(\Omega) $. Property ($p_2$) implies that $ A_p^{-1} $ is strongly continuous, and therefore  $A_p^{-1}(v_h) \to A_p^{-1}(v)$ almost everywhere in $\Omega$. 

Let now $ g\colon \R^+_0 \to \R^+_0 $ be defined by
\[
g(t) := a(bt)^{1/(p-1)} \quad \forall \, t \in \R^+_0,
\]
where the constants $ a $ and $ b $ come from  inequalities \eqref{1dis}-\eqref{2dis}. Note in particular that \eqref{1dis} holds true, since by assumption $p>N$. Clearly, $ g $ is monotone increasing in $ \R^+_0 $. Moreover, fix $ u \in U $. Then  property ($p_3$) gives
\[ 
\Vert u \Vert_{\infty} \leq a \Vert u \Vert = a \Vert A_p(u)\Vert_{W^{-1,p'}(\Omega)}^{1/(p-1)}\leq a(b\Vert A_p(u)\Vert_{p'})^{1/(p-1)} = g (\Vert A_p(u) \Vert_{p'}).
\]
This shows ($i_1$). Since hypotheses ($i_2$) and ($i_3$) are already satisfied, we have only to check ($i_4$). Define, for every $ x \in \Omega $,  
\[
\rho(x):= \sup_{|z| \le g(r)} d(0,F(x, z)). 
\]
Reasoning as in  \cite[Theorem 3.1]{marano3}, we  see that $ \Vert \rho \Vert_{p'} \le r $ once the same property holds true for the function $ x \mapsto j(x):= \sup_{|z| \le g(r)} \vert f(x, z) \vert $.


If $ \vert z \vert \le g(r) $, then
\[
\int_{\Omega} \vert f(x, z) \vert^{p'} dx \le m(\Omega) \Vert f(\cdot \,, z) \Vert_{\infty}^{p'} \, ,
\]
whence
\[
\begin{split}
\int_{\Omega} \vert j(x) \vert^{p'} dx&= \int_{\Omega} \biggl( \sup_{|z| \le g(r)} \vert f(x, z) \vert \biggr)^{p'} dx \le m(\Omega) \Vert f(\cdot \,, z) \Vert_{\infty}^{p'}.
\end{split}
\]
Choosing  $ r \ge m(\Omega)^{1/p'} \Vert f(\cdot \,, z) \Vert_{\infty}  $ gives $ j \in L^{p'}(\Omega) $ and $ \Vert j \Vert_{p'} \le r $, and hence hypothesis ($i_4$) is satisfied.

Thanks to Theorem \ref{plap} there exists $u \in U \subset W_0^{1,p}(\Omega)$ such that
\begin{equation}
\label{incl}
-\Delta_p u(x) \in F(x, u(x)) \quad \text{a.e. in }  \Omega
\end{equation}
and $ \vert \Delta_p u(x) \vert \le \rho(x) $ for almost every $ x \in \Omega $. 
Define $\Omega_f := \{ x \in \Omega: (x,u(x)) \in D \}$. From \eqref{D} it follows that
\[
\begin{split}
\Omega_f \subset & \, \bigl\{\pi_0(D_{\varphi}) \cap u^{-1}(\pi_1(D_{\varphi})) \bigr\} \\ 
& \quad  \cup \biggr \{\bigcup_{r \in D_k} \bigl[\pi_0(\varphi^{-1}(r) \setminus \text{int} (\varphi^{-1}(r)))\cap u^{-1} (\pi_1(\varphi^{-1}(r) \setminus \text{int} (\varphi^{-1}(r)))) \bigr] \biggr\},
\end{split}
\]
which, in particular, implies that
\[
\begin{split} m(\Omega_f) & \le m \l(\pi_0(D_{\varphi}) \cap u^{-1}(\pi_1(D_{\varphi}))\r) \\ 
& \qquad + m\l(\bigcup_{r \in D_k} [\pi_0(\varphi^{-1}(r) \setminus \text{int} (\varphi^{-1}(r)))\cap u^{-1} (\pi_1(\varphi^{-1}(r) \setminus \text{int} (\varphi^{-1}(r))))]\r) \\
& \le m \l(\pi_0(D_{\varphi}) \cap u^{-1}(\pi_1(D_{\varphi}))\r) \\ 
& \qquad + \bigcup_{r \in D_k} m\l([\pi_0(\varphi^{-1}(r) \setminus \text{int} (\varphi^{-1}(r)))\cap u^{-1} (\pi_1(\varphi^{-1}(r) \setminus \text{int} (\varphi^{-1}(r))))]\r).
\end{split}
\]
Assumption (ii) entails $ m(\pi_i(D_{\varphi}))= 0 $ for some $ i \in \{0, 1\} $. Likewise, due to (iii), for each $ r \in D_k $, there exists $ i_r \in \{0, 1\} $ such that $ m(\pi_{i_r}(\varphi^{-1}(r) \setminus \operatorname{int}(\varphi^{-1}(r))))= 0 $. Reasoning like in the case when $ \tilde \psi $ is constant gives $ m(\Omega_f)= 0 $. This implies $F(x,u(x)) = \{ f(x,u(x)) \}$ and   on account  of \eqref{incl} it follows that
\[
-\Delta_p u(x)= f(x,u(x)) \quad \text{a.e. in }  \Omega. 
\]
We then have
\[
\psi(-\Delta_p u(x))= \psi (f(x, u(x))) = \psi (k (\varphi (x, u(x)))) = \varphi (x, u(x)),
\]
which completes the proof.
\end{proof}

\begin{remark}
Hypothesis (iv) and the assumption $ 0 \notin (\alpha, \beta) $ are \emph{essential} to obtain the existence of a solution for equations as in \eqref{eq-disc}. Below we consider two situations:
apparently they are very similar, but one of them admits a solution while the other one doesn't. 
\end{remark}

\begin{example}

Let $ \varphi \colon \R \to \R $ be defined by
\[
\varphi(z)=
\begin{cases}
0 \, \,  & \text{if } \, z \ne 0 \\
1 & \text{if } \, z= 0.
\end{cases}
\]
and let $ \psi\colon [1, +\infty) \to \R $ be such that $ \psi(y)= y $. Consider the following equation
\begin{equation}
\label{ex}
-\Delta_p u= \varphi(u).
\end{equation}
Equation \eqref{ex} doesn't have any solution in $ W^{1,p}_0(\Omega) $. Suppose on the contrary  that $ u $ is such a solution. Since $ \varphi(u) \ge 0 $, then from  \eqref{ex} we have  $ -\Delta_p u \ge 0 $, and  the Strong Maximum Principle implies that $ u \equiv 0 $ or $ u> 0 $. If $ u \equiv 0 $, then this would imply that $ -\Delta_p u \equiv 0 $, which is in contrast with \eqref{ex}. Suppose now that $ u> 0 $. Then,  the definition of $ \varphi $ implies $ -\Delta_p u= 0 $. This fact, together with the boundary condition $ u  |_{\partial \Omega}= 0 $, implies $ u \equiv 0 $ which is again impossible. 

Observe also that such $\varphi $ is incompatible with the hypotheses of Theorem \ref{ultimo}, because in this case hypothesis (iv) and the condition $ 0 \notin (\alpha, \beta) $ cannot be verified simultaneously.


\medskip

Fix now $ \lambda \in (0, 1) $ and consider the function $ \tilde \varphi \colon \R \to \R $  defined by
\[
\tilde \varphi(z)=
\begin{cases}
1 \, \,  & \text{if } \, z \ne 0 \\
\lambda & \text{if } \, z= 0.
\end{cases}
\]
In this case both hypothesis (iv) and $ 0 \notin [1, +\infty) $ are verified, since 
\[
\{1\}= \overline{\tilde \varphi(\R \setminus \{0\})} \subset \psi([1, +\infty))= [1, +\infty).
\]
Therefore, Theorem \ref{ultimo} gives the existence of a solution $u \in W^{1,p}_0(\Om)$  to \eqref{ex}.

\end{example}

\end{document}